\documentclass[12pt]{amsart}
\usepackage[colorlinks=true, pdfstartview=FitV, linkcolor=blue, citecolor=blue, urlcolor=blue]{hyperref}
\usepackage{amssymb,amscd}
\calclayout 

\def\quot#1#2{#1/\!\!/#2}

\def\C{\mathbb {C}}

\def\R{\mathbb {R}}

\def\NN{\mathcal N}

\def\Z{\mathbb Z}

\def\SL{\operatorname{SL}}
\def\GL{\operatorname{GL}}

\def\Sp{\operatorname{Sp}}
\def\Spin{\operatorname{Spin}}
\def\SO{\operatorname{SO}}
\def\Orth{\operatorname{O}}

\def\inv{^{-1}}
\def\lie#1{{\mathfrak #1}}
\def\lieg{\lie g}
\def\lieh{\lie h}

\def\phi{{\varphi}}

\def\AA{\mathsf{A}}

\def\GG{\mathsf{G}}

\def\O{\mathcal O}

\def\HH{\mathcal H}

\def\pr{{\operatorname{pr}}}

\def\Lie{\operatorname{Lie}}
\def\Aut{\operatorname{Aut}}

\def\codim{\operatorname{codim}}

\def\rank{\operatorname{rank}}

\def\lieA{{\lie A}}

\def\an{{\operatorname{an}}}
\numberwithin{equation}{subsection}

\newtheorem{theorem}[subsection]{Theorem}
\newtheorem{lemma}[subsection]{Lemma}
\newtheorem{proposition}[subsection]{Proposition}
\newtheorem{corollary}[subsection]{Corollary}

\theoremstyle{definition}

\theoremstyle{remark}

\newtheorem{remark}[subsection]{Remark}

\title[Vector fields and Luna strata]{\boldmath Vector fields and Luna strata} 
 \author{Gerald W. Schwarz}
\address{Department of Mathematics\\
Brandeis University\\
Waltham, MA 02454-9110}
\email{schwarz@brandeis.edu}
\subjclass[2000]{20G20, 22E46, 22E60}
\keywords{Luna strata, vector fields}

\begin{document}
\begin{abstract}
Let $V$ be a $G$-module where $G$ is a complex reductive group. Let $Z:=\quot VG$ denote the categorical quotient.    One can ask  if the Luna stratification of $Z$ is intrinsic. That is, if $\phi\colon Z\to Z$ is any automorphism, does $\phi$ send strata to strata?
In Kuttler-Reichstein \cite{KuttlerReichstein}  the answer was shown to be yes for  $V$ a  direct sum of sufficiently many copies of a $G$-module $W$. We show that the answer is yes for almost all $V$. The key is to consider the vector fields on $Z$. Our methods also show that complex analytic automorphisms preserve the stratification.
 \end{abstract}

\maketitle

\section{Introduction}
Our base field is $\C$, the field of complex numbers. (Everything will also work for any algebraically closed field of characteristic zero, except for Remark \ref{rem:holomorphic} about holomorphic automorphisms.)\ Let $X$ be an irreducible smooth affine $G$-variety, e.g., a $G$-module. Here $G$ is reductive  but not necessarily connected. We denote the algebra of regular functions on $X$ by $\O(X)$. For the following, we refer to \cite{KraftBook}  and \cite{LunaSlice}. By Hilbert, the algebra $\O(X)^G$ is finitely generated, so that we have a quotient variety $Z:=\quot XG$ with coordinate ring $\O(Z)=\O(X)^G$. Let $\pi\colon X\to Z$ denote the morphism dual to the inclusion $\O(X)^G\subset\O(X)$. Then $\pi$ sets up a bijection between the points of $Z$ and the closed orbits in $X$. Let $x\in X$ such that the orbit $Gx$ is closed. Then the isotropy group $G_x$ is reductive and we may write $T_xX=T_x(Gx)\oplus N_x$ where $N_x$ is a $G_x$-module. We call the representation   $G_x\to\GL(N_x)$ the \emph{slice representation of $G_x$\/}. The isotropy stratum $Z_{(H)}$ of $Z$ consists of the closed orbits whose isotropy groups are conjugate to the reductive subgroup $H$ of $G$. The Luna strata of $Z$ consist  of the  irreducible components of the various $Z_{(H)}$ such that the  corresponding modules $(N_x,G_x)$ are isomorphic (after conjugating $G_x$ to $H$) to a fixed $H$-module $W$. If $X$ is a $G$-module, then isotropy strata are irreducible and there is no difference between Luna strata and isotropy strata. From now on we will call the irreducible components of the isotropy (or Luna) strata of $Z$ simply the \emph{strata of $Z$}. The strata are locally closed, smooth and finite in number.

Let $(H)$ and $(K)$ be conjugacy classes of reductive subgroups of $G$. We write $(H)< (K)$ if $H$ is conjugate to a proper subgroup of $K$. We say that $(H)$ is an isotropy class of $Z$ (or $X$) if $Z_{(H)}\neq\emptyset$. Then among the isotropy classes there is a minimum element $(H)$ called the \emph{principal isotropy class\/}. We call any corresponding closed orbit a \emph{principal orbit\/}. The corresponding stratum $Z_\pr$ of $Z$ is open and dense. In terms of slice representations, $(H)$ is the unique isotropy class such that the corresponding slice representation is of the form $H\to\GL(W+\Theta)$ where $\Theta$ denotes a trivial $H$-module and $W$ is an $H$-module with $\O(W)^H=\C$.   We  denote $\pi\inv(Z_\pr)$ by $X_\pr$.

As shorthand for saying that $X$ has finite principal isotropy groups (resp.\ trivial principal isotropy groups) we say that $X$ has FPIG (resp.\ TPIG).
We say that $X$ is \emph{stable\/} if there is a nonempty open subset of closed orbits, equivalently; the slice representation of the principal isotropy group is trivial.  
We say that $X$ is \emph{$k$-principal\/} if $\codim_X (X\setminus X_\pr)\geq k$.  

Let $\lieA(X)$ denote the derivations of $\O(X)$, equivalently; the algebraic vector fields on $X$. Let $\lieA(Z)$ denote the derivations of $\O(Z)$ and let $\pi_*\colon \lieA(X)^G\to\lieA(Z)$ be the restriction morphism. That is, $\pi_*(A)(f)=A(f)$ for $A\in\lieA(X)^G$ and $f\in\O(Z)\simeq\O(X)^G$.

\begin{theorem}\label{thm:lunainvar}
If 
$\pi_*\colon \lieA(X)^G\to\lieA(Z)$ is surjective, then the stratification of $Z$ is intrinsic. 
\end{theorem}

The following theorem sharpens  results of \cite{KuttlerReichstein} and \cite{Kuttler}.

\begin{theorem}\label{thm:adjointrep}
Let $G$ be a simple algebraic group. Let $V$ be $r$  copies of the adjoint module of $G$. Then the Luna stratification of $Z:=\quot VG$ is intrinsic if   $G=\AA_1$ and $r\geq 3$  or 
$G\neq\AA_1$ and $r\geq 2$.
\end{theorem}

We should mention the case where the quotient $Z$ is smooth (then $V$ is called \emph{coregular\/}). In this case, the stratification is intrinsic if and only if there is only one stratum. This is precisely the case where $V$ is \emph{fix-pointed\/}, i.e., the closed $G$-orbits are the fixed points. Then $V=V^G\oplus V'$ as $G$-module where $\O(V')^G=\C$. If $V$ is coregular and not fix-pointed, then the stratification of $Z$ is never intrinsic. Now it follows that the theorem above is best possible since $V$ is coregular when  $r=1$ and when $r=2$ and $G=\AA_1$. 

Let $X$ have FPIG and let $X_{(n)}$ denote the points of $X$ with isotropy group of dimension $n$. Then the transcendence degree of $\C(X_{(n)})^G$ is $c_n:=\dim X_{(n)}-\dim G+n$. We say that $X$ is \emph{$k$-modular\/} if $c_n \leq\dim Z-k$ for any $n\geq 1$.   We say that $X$ is \emph{$k$-large\/} if $X$ is $k$-principal and $k$-modular.

\begin{theorem} \label{thm:sufficient} Assume that $X$ has FPIG. Then the map $\pi_*\colon\lieA(X)^G\to\lieA(Z)$ is surjective in the following cases.
\begin{enumerate}
\item $X$ is $2$-principal and every slice representation of $X$ is orthogonal \cite[Theorem 0.2, Proposition 3.5]{SchLifting}.
\item $X$ is $3$-principal   \cite[Theorem 8.9]{SchLiftingDOs}.
\item $X$ is $2$-large \cite[Theorem 9.10]{SchLiftingDOs}.
\end{enumerate}
Hence if any of the conditions is satisfied, then the stratification of $Z$ is intrinsic.
\end{theorem}

Let $\lieg$ denote  the Lie algebra of $G$. Suppose that $G=G_1\times G_2$ is a product of reductive groups. Consider modules $V=V_1\oplus V_2$ where $V_i$ is a  $G_i$-module, $i=1$, $2$. If  the  stratification of $Z_1=\quot {V_1}{G_1}$ is not intrinsic, then the   stratification of $\quot VG$ isn't either, for any $V_2$. This explains why the hypotheses of the following theorem are necessary.

\begin{theorem}\cite[Corollary 11.6]{SchLiftingDOs}\label{thm:finitebad}
Let $G$ be a connected semisimple group and consider representations of $G$ which contain no trivial factor and all of whose irreducible factors are faithful representations of $\lieg$. Then, up to isomorphism, there are only finitely many such representations which do not have FPIG or are not $2$-large. Hence there are only finitely many representations, up to equivalence, for which the stratification of $Z$ is not intrinsic.
\end{theorem}

In \S\ref{sec:preservestratif} we establish Theorem \ref{thm:lunainvar}. We discuss the relations between the various hypotheses placed on $X$ in  Theorem \ref{thm:sufficient}. We mention some interesting families of $G$-modules where the following alternative holds.  Either $V$ is coregular  or $V$ satisfies one of the  hypotheses of Theorem \ref{thm:sufficient}.    In \S\ref{sec:adjoint} we establish Theorem \ref{thm:adjointrep}.  
 
 We thank the referees and Stephen Donkin for helpful comments.
 
 \section{Preserving the stratification}\label{sec:preservestratif}  We need some preliminaries on Luna's slice theorem in order to establish Theorem \ref{thm:lunainvar}.
  Let $P$ and $Q$ be affine $G$-varieties and let $\tau\colon P\to Q$ be equivariant. Then we have the induced mapping $\quot\tau G\colon \quot PG\to\quot QG$. Following Luna we say that $\tau$ is \emph{excellent} if
 \begin{enumerate}
\item $\tau$ is \'etale.
\item $\quot \tau G$ is \'etale.
\item The canonical morphism $(\pi,\tau)\colon P\to \quot PG\times_{\quot QG} Q$ is an isomorphism.
\end{enumerate}
Recall that \'etale morphisms are those which are   smooth with finite fibers. In case $P$ and $Q$ are smooth, $\tau$ is \'etale if and only if $d\tau_p\colon T_pP\to T_{\tau(p)}Q$ is an isomorphism for all $p\in P$. Let $\tau$ be excellent. Then  $\tau$ is \emph{isovariant\/}, i.e., $G_p=G_{\tau(p)}$ for all $p\in P$.
If $\tau$ is also surjective then $\quot\tau G$ sends $(\quot PG)_{(H)}$ onto  $(\quot QG)_{(H)}$ for every $H$.    

Assume that $P$ and $Q$ are smooth. Let $A\in\lieA(Q)$. We may consider  $A$ as a family of tangent vectors $A(q)\in T_qQ$, $q\in Q$. Define $\tau^*A$ by $\tau^*(A)(p)=d\tau_p\inv A(\tau(p))$, $p\in P$. Then $\tau^*A\in\lieA(P)$ (\cite[4.2]{SchLiftingDOs} or \cite{Levasseur}) and we have

\begin{lemma} \label{lem:excellent} Let $\tau$ be as above. Then  $\tau^*$ gives rise to an  isomorphism  of $\O(P)\otimes_{\O(Q)}\lieA(Q)$ with $\lieA(P)$.  If $A\in\lieA(Q)^G$, then $\tau^*A\in\lieA(P)^G$ and we also have an induced isomorphism   of $\O(P)^G\otimes_{\O(Q)^G}\lieA(Q)^G$  with $\lieA(P)^G$. 
\end{lemma}

 Recall that a subset $Y$ of $X$ is \emph{$G$-saturated\/} if $Y=\pi\inv(\pi(Y))$. Let $S$ be a stratum of $Z=\quot XG$. For $z\in Z$ let $\lieA(Z)(z)$ denote the values at $z$ of elements of $\lieA(Z)$. We say that    
 $\lieA(Z)$    \emph{preserves   $S$\/} if  $\lieA(Z)(s)\subset T_sS$ for every $s\in S$ and we say that $\lieA(Z)$   \emph{spans $S$\/} if $\lieA(Z)(s)\supset T_sS$ for every $s\in S$.

 \begin{proposition}\label{prop:preserve}
 Suppose that $\pi_*\lieA(X)^G=\lieA(Z)$. Then $\lieA(Z)$ spans and preserves every stratum of $Z$.
 \end{proposition}
 
\begin{proof} Let $x\in X$ such that $Gx$ is closed and set $H=G_x$. We use Luna's slice theorem \cite{LunaSlice} to find the local structure of $Z$ and $S=Z_{(H)}$ near $Gx$. Now  there is a slice at $x$, that is, a locally closed $H$-saturated   smooth subvariety $Y$ of $X$ containing $x$ such that $G\times^HY\to X$ is excellent. Moreover, we can assume that $U:=GY$ is affine and $G$-saturated. Then 
\begin{equation}\label{eq:eq1}
\lieA(U)^G\simeq\O(U)^G\otimes_{\O(X)^G}\lieA(X)^G
\end{equation}
 and by Lemma \ref{lem:excellent} we have that 
 \begin{equation}\label{eq:eq2}
  \lieA(G\times^HY)^G\simeq\O(G\times^HY)^G\otimes_{\O(U)^G}\lieA(U)^G
 \end{equation}
 where, of course, $\O(G\times^HY)^G\simeq\O(Y)^H$.
Again, by Luna,   there is an excellent   $H$-morphism  $\tau\colon Y\to N$ with image $N_f$, where $N\simeq T_x(Y)$ and $f\in\O(N)^H$ does not vanish at the origin. Here $T_xX=T_x(Gx)\oplus N$ as $H$-module. Thus we have
\begin{equation}\label{eq:eq3}
\lieA(G\times^HY)^G\simeq\O(Y)^H\otimes_{\O(N_f)^H}\lieA(G\times^H N_f)^G.
\end{equation}
We may write $N=N^H\oplus N'$ as $H$-module. Then the $(H)$-stratum   of $\quot {(G\times^HN_f)}G\simeq\quot {N_f}H$ has inverse image $(N^H\times\NN(N'))_f$ where $\NN(N')$ is the null cone of $N'$. Now the $H$-invariants of $N^H\times N'$ are generated by the coordinate functions of $N^H$ and $\O(N')^H$ where every element of $\O(N')^H$ has vanishing differential at the origin. Let $\rho\colon N\to\quot NH$ be the quotient morphism.  Since $N\simeq N^H\oplus N'$ and $\quot NH\simeq N^H\times\quot{N'}H$, the differential of $\rho$ at $0$,   $d\rho_0$,   has rank $\dim N^H$ and $d\rho_0(N^H)=T_{\rho(0)}(\quot {N_f}H)_{(H)}$. Now going back to $X$ we see that $d\pi_x\colon T_xX\to T_{\pi(x)}Z$ has image the tangent space at $\pi(x)$ to the stratum $Z_{(H)}$.  

Now it remains to show that $\lieA(X)^G$ evaluated at $x$, modulo $T_x(Gx)$, maps onto $N^H$. But for every $v\in N^H$ there is the ($H$-invariant) constant  vector field $A$ with value $v$.  The vector field $A$ induces a  $G$-invariant vector field on $G\times^H N_f$. It  then follows from equations \eqref{eq:eq1}--\eqref{eq:eq3}  that $\lieA(X)^G$ evaluated at $x$ projects onto $N^H$. Thus $\lieA(Z)=\pi_*\lieA(X)^G$ preserves and spans the strata of $Z$.
\end{proof}

\begin{corollary}[= Theorem \ref{thm:lunainvar}]\label{cor:thmmain}
Suppose that $\pi_*\lieA(X)^G=\lieA(Z)$. Then the stratification of $Z$ is intrinsic.
\end{corollary}

\begin{proof}  Let $\phi\in\Aut(Z)$, let $z\in Z$ and let $A\in\lieA(Z)$,  considered   as a derivation of $\O(Z)$. Since $\phi^*$ is an automorphism of $\O(Z)$, $\phi^*(A):= \phi^*\circ A\circ(\phi^*)\inv$ is again a derivation of $\O(Z)$.  In terms of tangent vectors we have that 
$$
\phi^*(A)(z)=d(\phi\inv)_{\phi(z)}A(\phi(z))=(d\phi_z)\inv A(\phi(z)).
$$
Hence $\dim\lieA(Z)(z)=\dim\lieA(Z)(\phi(z))$. Let
$S$ denote the  stratum of $Z$ containing $z$ and set $k:=\dim S$.
Since  $\dim\lieA(Z)(\phi(z))=k$, $\phi(z)\in S'$ where $S'$ is a stratum of dimension $k$. Since the strata are irreducible and finite in number, we find that $\phi(S)\subset S'$. Hence $\phi$ preserves the  stratification.
\end{proof}
 
  \begin{remark}\label{rem:holomorphic}
 Let $\lieA_\an(Z)$ denote the derivations of the holomorphic functions $\HH(Z)$ on $Z$. In the case that $X$ is a $G$-module,   \cite[6.1, 6.6 and 6.9]{SchLifting}  show the following. 
 \begin{enumerate}
\item $\lieA_\an(X)^G$ is generated over $\HH(X)^G$ by $\lieA(X)^G$.
\item $\lieA_\an(Z)$ is generated over $\HH(Z)$ by $\lieA(Z)$. 
\item The natural map $\pi_*\colon\lieA_\an(X)^G\to\lieA_\an(Z)$ is surjective if and only if $\pi_*\colon\lieA(X)^G\to\lieA(Z)$ is surjective.
\end{enumerate}
Using the slice theorem, one can establish the same results in case that $X$ is an irreducible smooth affine $G$-variety.
Hence the argument of Corollary \ref{cor:thmmain} shows that the strata of $Z$ are preserved by  holomorphic automorphisms of $Z$.
 \end{remark}

 We now make some remarks concerning the conditions which arise in Theorem \ref{thm:sufficient}
\begin{remark} 
\begin{enumerate}
\item Let $H$ be reductive and let $W$ be a $H$-module which is not fix-pointed such that $\dim\quot WH=1$. Then $\quot WH\simeq\C$ and $\pi_*\colon\lieA(W)^H\to \lieA(\quot WH)$ is not surjective since all elements of $\lieA(W)^H$ vanish at zero while this is false for elements of $\lieA(\C)$. It then follows from the slice theorem that $\pi_*\colon \lieA(X)^G\to\lieA(Z)$ is not surjective if $X$ has a slice representation $H\to\GL(W)$ as above, i.e., $\pi_*$ is not surjective in case that $Z$ has a codimension one stratum. Now suppose that $V$ is a $G$-module where $G$ is connected semisimple or $V$ is orthogonal. Then the inverse image of a stratum of codimension two (or more) is of codimension two \cite[Corollary 7.4]{SchLifting}. Thus if there are no codimension one strata, $V$ has to be $2$-principal. Thus in many cases, surjectivity of $\pi_*$  implies that $V$ is $2$-principal.
\item If $G$ is simple or $V$ is an irreducible $G$-module and $G$  is semisimple, then $\pi_*$ is not surjective in case  $V$ does not have FPIG and is not fix-pointed   \cite[Theorem 7.13]{SchLiftingDOs}. Thus we are more or less forced to assume that $X$ has FPIG if we want $\pi_*$ to be surjective.
\end{enumerate}
\end{remark} 

  \begin{remark}\label{rem:TPIG}
Suppose that the $G$-module $V$ is  is $2$-principal and stable. Note that  $V$ is stable if it  has FPIG  or is orthogonal (\cite{LunaClosed} or \cite[Cor.\ 5.9]{SchLifting}). Let $H$ be a principal isotropy group. Then by \cite[Theorem 7.5]{SchLiftingDOs}, we have an isomorphism $G\times ^{N_G(H)} V^H\simeq V$. Since $V$ is acyclic, so are $G\times^{N_G(H)} V^H$ and $G/N_G(H)$. Let $K$ (resp.\ $L$) be maximal compact subgroups of $G$ (resp.\ $N_G(H)$) where $L\subset K$. Then $G/N_G(H)$ retracts onto $K/L$. Since $K/L$ is a compact manifold with the $\Z/2\Z$-cohomology of a point, we must have that $K/L$ is a point. Thus $N_G(H)=G$, $H$ is normal in $G$ and we must have that $V^H=V$. Hence $H$ is the kernel of the action of $G$ on $V$. Replacing $G$ by its image in $\GL(V)$ we see that our representation actually has TPIG.
\end{remark}

As we remarked before, if $V$ is coregular and not fix-pointed, there is no hope of the   stratification being intrinsic. There turn out to be  interesting series of $G$-modules $V$ which have the property that either $V$ is coregular or $\pi_*$ is surjective. In the following theorem, $k$, $\ell\geq 0$ and $(\C^7,\GG_2)$ and $(\C^8,\Spin_7)$ are the unique irreducible modules of the indicated dimensions.

\begin{theorem}\label{thm:LS}\label{rem:2large}
Let $V$ be one of the following  $G$-modules. 
\begin{enumerate}
\item $(V,G)=(k\C^n+\ell (\C^n)^*,\GL_n)$.\label{eq:gln}
\item $(V,G)=(k\C^n,\SO_n)$ or $(k\C^n,\Orth_n)$.\label{eq:sonon}
\item $(V,G)=(k\C^{2n},\Sp_{2n})$.\label{eq:spn}
\item Any  $\SL_2$-module.
\item $(V,G)=(k\C^7,\GG_2)$ or $(k\C^8,\Spin_7)$.\label{eq:g2b3}
\item Any irreducible nontrivial  module of a simple algebraic group.\label{eq:simple} 
\item $(V,G)=(k\C^n,\SL_n)$.\label{eq:sln}
\end{enumerate}
 If $V$ is not coregular, then $\pi_*$ is surjective and the stratification of $Z$ is intrinsic.
With the exception of case \eqref{eq:sln},  the module is $2$-large if it is not coregular. In  case \eqref{eq:sln}, when $n+2\leq k<2n$, the module is not coregular and not $2$-large, but it is $3$-principal. For $k\geq 2n$ it is $2$-large.
\end{theorem}

\begin{proof}
Cases \eqref{eq:gln}--\eqref{eq:spn} were first established by  Levasseur and Stafford \cite{LS}. In \cite[Chapter 11]{SchLiftingDOs} it is shown that $V$ is $2$-large if $V$ is not coregular in cases \eqref{eq:gln}--\eqref{eq:g2b3} and \cite{SchDiffSimple} establishes the same thing for case \eqref{eq:simple}. The claims about case \eqref{eq:sln} can be found in \cite[Chapter 11]{SchLiftingDOs}.
\end{proof}

Let $(H_1)$ and $(H_2)$ be isotropy classes of closed orbits in $V$.   Then $(H_1)<(H_2)$ implies that $Z_{(H_2)}$ lies in the closure of $Z_{(H_1)}$, hence the strata have different dimensions. 

\begin{remark}
 In   cases \eqref{eq:gln}, \eqref{eq:sonon}, \eqref{eq:spn}, \eqref{eq:g2b3} and \eqref{eq:sln} above,  one can check that  the isotropy classes of closed orbits form a chain $(H_1)<(H_2)<\dots$ so that the various strata have different dimensions. Thus when the stratification is intrinsic, so is each stratum.
\end{remark}

Let $H$ be a subgroup of $G$ and $V$ a $G$-module. We use the notation $(V,G)$ (resp.\ $(V,H)$) to specify which group is acting on $V$.
   \begin{lemma}
   Let $H$ be a reductive subgroup of $G$ and let $V$ be a $G$-module. 
   \begin{enumerate}
\item If $(V,G)$ has FPIG (resp.\ TPIG), then so does $(V,H)$.
\item If $(V,G)$ has TPIG and is $k$-principal, then the same holds for $(V,H)$.
\item If $(V,G)$ is $k$-modular, then so is $(V,H)$.
\end{enumerate}
      \end{lemma}
 \begin{proof}
Parts (1) and (2) are trivial. One needs only to note that if $Gv$ is closed with $G_v$ trivial (resp.\ finite), then $Hv$ is closed with $H_v$ trivial (resp.\ finite).
For (3) note that we have  FPIG, by definition.  Let $V_{(n)}=\{v\in V\mid \dim G_v=n\}$ and let $V_{(n,m)}=\{v\in V_{(n)}\mid \dim H_v=m\}$. We have to show that, for $n$, $m\geq 1$, 
$\dim V_{(n,m)}-\dim H+m\leq\dim\quot VH-k.$ But
\begin{align*}
&\dim V_{(n,m)}-\dim H+m\leq\dim V_{(n)}-\dim G+n+\dim G-\dim H\leq\\
&\leq \dim \quot VG-k+\dim G-\dim H=\dim\quot VH-k.
\end{align*}
Hence $(V,H)$ is $k$-modular.
 \end{proof}

 Suppose that $G\subset\GL(W)$ where $G^0\subset\SL(W)$.    Consider $V=rW$, the direct sum of $r$ copies of $W$. We have an improvement  of  \cite[Theorem 1.2]{KuttlerReichstein}. 

\begin{theorem}
Let $V=rW$ as above, let $\pi\colon V\to\quot VG$ be the quotient morphism and let $n=\dim W$.
\begin{enumerate}
\item If $r\geq n+2$, then $\pi_*$ is surjective.
\item If $W$ is orthogonal and $r\geq n+1$, then $\pi_*$ is surjective.
\end{enumerate}
Hence in both cases the stratification of $Z$ is intrinsic.
\end{theorem}
 
 \begin{proof}
For (1) we leave the case $n=1$ to the reader, so assume that $n\geq 2$. Replace $G$ by the group $H$ generated by $G$ and $\SL(W)$. Then $H$  is the extension of $\SL(W)$ by a finite group of scalars. The action of $\SL(W)$ on $V$ is $3$-principal with TPIG, and the action of the group of scalars on $\quot V{\SL(W)}$ is by the induced scalar action. Since the dimension of $\quot V{\SL(W)}$ is at least $5$, the action of $H$ is 3-principal with TPIG. Hence so is the action of $G$ and we have (1).

 For   (2) note that the action of $\Orth(W)$ on $(n+1)W$ has TPIG and is $2$-large (Theorem \ref{thm:LS}\eqref{eq:sonon}), hence so is the action of $G$.    
 \end{proof}
 The theorem is best possible since the modules $((n+1)\C^n,\SL_n)$ and $(n\C^n,\Orth_n)$ are coregular.
  
 \section{The adjoint representation}\label{sec:adjoint}
 The following result is established for type $\AA$ in \cite{LeBruynProcesiEtale}.
 \begin{proposition}\label{prop:adjoint}
 Let $V$ be $r$ copies of the adjoint module of the simple adjoint algebraic group $G$. Then $Z$ has no codimension one strata if (and only if) we have
 \begin{enumerate}
\item $r\geq 3$ and $G=\AA_1$.
\item $r\geq 2$ and $\rank G\geq 2$.
\end{enumerate}
 \end{proposition} 
 
 \begin{proof}
 We leave the case $G=\AA_1$ to the reader, so suppose that $\rank G\geq 2$. It is clearly enough to prove that there are no codimension one strata in the case that $r=2$. Now the principal isotropy group of $G$ acting on $\lieg$ is a maximal torus $T$ and the slice representation is the action of $T$ on $\lieg$, plus a trivial representation. Since the action of $T$ is faithful, the principal isotropy group is trivial. Thus if there is a codimension one stratum we need to have an orthogonal slice representation of the form $H\to\GL(W+\Theta)$ where $W$ has trivial principal isotropy group and a one-dimensional quotient.  Thus $(W,H)$ is the complexification of  a module $(W_\R,H_\R)$ of a compact Lie group  $H_\R$ which acts freely and  transitively on the sphere of $W_\R$. Hence $H_\R$ is $S^0$, $S^1$ or $S^3$ and we have the following possibilities.
 \begin{enumerate}
\item $(W,H)=(\C,\pm 1)$.
\item $(W,H)=(\nu_1+\nu_{-1},\C^*)$.
\item $(W,H)=(2\C^2,\SL_2)$.
\end{enumerate}
Here $\nu_j$ denote the one-dimensional $\C^*$ module of weight $j$.
Recall that, up to trivial factors, $V=\lieg/\lieh\oplus W$ as $H$-module. In case (1), since $G\subset\SL(\lieg)$, we must have $H\subset\SL(W)$, which is a contradiction.  Case (3) is ruled out by \cite[Corollary 13.4]{SchLifting}. Thus we are left with case (2). Since $V=\lieg+\lieg=\lieg/\lieh+W$ where $\lieh$ is a trivial module, we see that the action of $H=\C^*$ on $\lieg$ is $\nu_1+\nu_{-1}+\Theta$. Let $A$ generate  $\Lie(H)$ and choose a maximal torus $T$ and a set of simple roots. Since $H$ is conjugate to a subgroup of  $T$, our condition says that $\alpha(A)=0$ for all roots $\alpha$ of $\lieg$, except for at most one pair $\pm\beta$. We can assume that $\beta=\alpha_1$ is one of the simple roots $\alpha_1,\dots,\alpha_\ell$. Since $\ell>1$ we may assume that $\alpha_1+\alpha_2$ is also a root. Thus $\alpha_1+\alpha_2$, $\alpha_2,\dots,\alpha_\ell$ vanish on $A$. Hence $A=0$, a contradiction. Thus there are no codimension one strata.
  \end{proof}
 
 \begin{proof}[Proof of Theorem \ref{thm:adjointrep}]   This follows from the proposition above and Theorem \ref{thm:sufficient}(1). \end{proof}


 \providecommand{\bysame}{\leavevmode\hbox to3em{\hrulefill}\thinspace}
\providecommand{\MR}{\relax\ifhmode\unskip\space\fi MR }
\providecommand{\MRhref}[2]{
  \href{http://www.ams.org/mathscinet-getitem?mr=#1}{#2}
}

\providecommand{\href}[2]{#2}

   \end{document}